\documentclass[12pt,a4paper]{amsart}

\usepackage{amsmath}
\usepackage{amssymb}

\theoremstyle{plain}   
\begingroup 
\newtheorem*{th_PZ}{Theorem PZ}
\newtheorem*{th_I}{Theorem I}
\newtheorem*{cor_PZ}{Corollary PZ}
\newtheorem{theorem}{Theorem}[section]   
\newtheorem{corollary}[theorem]{Corollary}     
\newtheorem{lemma}[theorem]{Lemma}         
\newtheorem{proposition}[theorem]{Proposition}  
\newtheorem{fact}[theorem]{Fact}
\endgroup

\theoremstyle{definition}
\newtheorem{definition}[theorem]{Definition}   

\theoremstyle{remark}
\newtheorem{remark}[theorem]{Remark}        

\numberwithin{equation}{section}

\newcommand{\ep}{\varepsilon}

\newcommand{\R}{{\mathbb R}}
\newcommand{\N}{{\mathbb N}}

\newcommand{\dom}{\operatorname{dom}}

\newcommand{\vf}{\varphi}

\newcommand{\interior}{\operatorname{int}}

\usepackage[dvips]{color}

\begin{document}

\title[G\^ ateaux and Hadamard differentiability]{G\^ ateaux and Hadamard differentiability via directional differentiability}

\thanks{The research was supported by the grant GA\v CR P201/12/0436.}

\author{Lud\v ek Zaj\'\i\v{c}ek}



\email{zajicek@karlin.mff.cuni.cz}

\address{Charles University,
Faculty of Mathematics and Physics,
Sokolovsk\'a 83,
186 75 Praha 8-Karl\'\i n,
Czech Republic}


\begin{abstract} 
Let $X$ be a separable Banach space, $Y$ a Banach space and $f: X \to Y$ an arbitrary mapping. 
Then the following implication holds at each point $x \in X$ except a $\sigma$-directionally porous set:\  
 If the one-sided Hadamard directional derivative $f'_{H+}(x,u)$ exists in all directions $u$ from a set $S_x \subset X$
  whose linear span is dense in $X$, then $f$ is Hadamard differentiable at $x$.
    This theorem improves and generalizes a recent result of A.D. Ioffe, in which the linear span of $S_x$ equals $X$ and $Y = \R$.  
    An analogous theorem, in which $f$ is pointwise Lipschitz, and which deals with the usual one-sided derivatives and G\^ ateaux differentiability is also proved.  It generalizes a result of D. Preiss and the author, in which $f$ is supposed to be Lipschitz.
    \end{abstract}
   

\maketitle

\section{Introduction}

The following result (\cite[Theorem 5]{PZ}) was used in the proof of the till now strongest  version of Rademacher's theorem on G\^ ateaux differentiability of Lipschitz mappings on a separable Banach space.

\begin{th_PZ}\label{th_PZ}
Let $X$ be a separable Banach space, $Y$ a Banach space, $G \subset X$ an open set, and $f: G \to Y$ a Lipschitz mapping. Then there exists a $\sigma$-directionally porous set $A \subset G$ such that for every $x \in G \setminus A$
 the set $U_x$ of those directions $u \in X$ in which the one-sided derivative $f'_+(x,u)$ exists is a closed linear subspace of $X$. Moreover, the mapping $u \mapsto f'_+(x,u)$ is linear on $U_x$. 
\end{th_PZ}

An immediate consequence of this result is the following.

\begin{cor_PZ}\label{cor_PZ}
Let $X$ be a separable Banach space, $Y$ a Banach space, $G \subset X$ an open set, and $f: G \to Y$ a Lipschitz mapping. Then the following implication holds at each point $x \in G$ except a $\sigma$-directionally porous set:
  
 If the one-sided directional derivative $f'_{+}(x,u)$ exists in all directions $u$ from a set $S_x \subset X$
  whose linear span is dense in $X$, then $f$ is G\^ ateaux differentiable at $x$.
\end{cor_PZ}
 
   Note that  each  $\sigma$-directionally porous subset of a
 separable Banach spaces $X$ is   not only a first category set, but it is also ``measure null'': it is Aronszajn (=Gauss) null, and so also Haar null, (see \cite[p. 164 and  Chap. 6]{BL}) and
  also $\Gamma$-null (see \cite[Remark 5.2.4] {LPT}). 
  
  It is an easy  well-known fact that if $f:X \to Y$ is a Lipschitz mapping between Banach spaces, then the Hadamard (one-sided) directional derivatives coincide with the usual (one-sided) directional derivatives and the Hadamard derivative coincides with the G\^ ateaux derivative.
  
  A.D. Ioffe  in \cite{Io}  recently observed that some known results dealing with the usual directional derivatives
   (and G\^ ateaux differentiability) of {\it Lipschitz} functions can be generalized to results dealing with the Hadamard directional derivatives (and Hadamard differentiability) of {\it arbitrary} functions. (Note that this Ioffe's idea
    was followed in \cite{Za1}.)

  Ioffe's \cite[Theorem 3.7(b)]{Io} can be reformulated (see Remark \ref{dom}  below) in the following way.
  
  \begin{th_I}\label{th_I}
  Let $X$ be a separable Banach space,  $G \subset X$ an open set, and $f: G \to \R$ an arbitrary function. Then the following implication holds at each point $x \in G$ except a $\sigma$-directionally porous set:
   
  If the one-sided Hadamard directional derivative $f'_{H+}(x,u)$ exists in all directions $u$ from a set $S_x \subset X$
  whose linear span equals to $X$, then $f$ is Hadamard differentiable at $x$.
   \end{th_I}
   
   So, Theorem I is a partial generalization of Corollary PZ (since in Theorem I
    is a stronger assumption on $S_x$ and $Y= \R$). We will prove (see Corollary \ref{hadcom} below) that the corresponding full generalization of Corollary PZ holds. 
    
    Moreover, Theorem \ref{had} below  on Hadamard derivatives generalizes Theorem PZ.

      Further,  we generalize Theorem PZ in another direction showing that, in this theorem,
      it is sufficient to suppose that $f$ is pointwise Lipschitz (see Corollary \ref{plux} below).

      Note  that the methods of proofs from \cite{PZ} and \cite{Io} cannot be easily used in the case when $f: X \to Y$ is not continuous and $Y$ is nonseparable, since then $f(X)$ need not  be separable. From this reason we use an alternative method based on a small trick in the proof of Lemma \ref{rozd} (which shows that $f'_+(x,w_x)$ and $f'(y_k, w_{y_k})$ are ``automatically'' close to one another). 
      
    The main results are proved in Section 3. 
    In Section 2 we recall basic definitions and  known results we need.

 \section{Preliminaries} 
 
 In the following, by a Banach space we mean a real Banach space. 
  The symbol $B(x,r)$ will denote (in a metric space) the open ball with center $x$ and radius $r$.
  
 If $X$ is a Banach space, we set 
    $S_X:= \{x \in X: \|x\|=1\}$. 
    Further, if $x \in X$, $v \in S_X$ and $\delta >0$, then we define the open {\it cone
     $C(x,v,\delta)$} as the set of all $y \neq x$ for which  $\|v - \frac{y-x}{\|y-x\|}\| < \delta$.

    Let $X$, $Y$ be Banach spaces, $G \subset X$ an open set, and $f:G \to Y$ a mapping. 
    
   We say  that {\it $f$ is Lipschitz at $x \in G$} if  $\limsup_{y \to x} \frac{\|f(y)-f(x)\|}{\|y-x\|} < \infty$.
   We say  that $f$ is {\it pointwise Lipschitz} (on $G$) if $f$ is Lipschitz at all points of $G$.

   The directional and one-sided  directional derivatives 
    of $f$ at $x\in G$ in the direction $v\in X$ are defined respectively by
   $$f'(x,v) := \lim_{t \to 0} \frac{f(x+tv)-f(x)}{t}\ \ \text{and}\ \  f'_+(x,v) := \lim_{t \to 0+} \frac{f(x+tv)-f(x)}{t}.$$
    The {\it Hadamard directional and one-sided  directional derivatives} 
    of $f$ at $x\in G$ in the direction $v\in X$ are defined respectively by
   $$f'_H(x,v) := \lim_{z \to v, t \to 0} \frac{f(x+tz)-f(x)}{t}\ \ \text{and}\ \  f'_{H+}(x,v) := \lim_{z \to v, t \to 0+} \frac{f(x+tz)-f(x)}{t}.$$
   
   The following facts are well-known and easy to prove.
   
   \begin{fact}\label{f}
    Let $X$, $Y$ be Banach spaces, $G \subset X$ an open set, $x \in G$, $v \in X$,  and $f:G \to Y$ a mapping. 
    Then the following assertions hold.
    \begin{enumerate}
    \item
    The derivative $f'(x,v)$ (resp. $f'_H(x,v)$) exists if and only if  $f'_+(x,-v) = -f'_+(x,v)$
    (resp. $f'_{H+}(x,-v) = - f'_{H+}(x,v)$). 
    \item
    If  $f'_H(x,v)$ (resp. $f'_{H+}(x,v)$) exists, then $f'(x,v)=f'_H(x,v)$ (resp. $f'_+(x,v)=f'_{H+}(x,v)$).
    \item
     If $f$ is locally Lipschitz on $G$, then  $f'(x,v)=f'_H(x,v)$ (resp. $f'_+(x,v)=f'_{H+}(x,v)$) whenever one of these two derivatives exists.
     \item
     If $f'_+(x,v)$ (resp. $f'_{H+}(x,v)$) exists and $t>0$, then $f'_+(x,tv)= t f'_+(x,v)$ (resp. $f'_{H+}(x,tv)= t f'_{H+}(x,v)$).
     \item
     If $f$ is Lipschitz at $x$ with $\limsup_{y \to x} \frac{\|f(y)-f(x)\|}{\|y-x\|}\leq K < \infty$ and 
      $f'_+(x,v)$ exists, then  $\|f'_+(x,v)\| \leq K \|v\|$.
      \item
      If $V_x$ is the set of all $u \in X$, for which $f'_{H+}(x,u)$ exists, then the mapping $u \mapsto f'_{H+}(x,u)$
       is continuous on $V_x$.
       \item
       If $v \in S_X$ and $f'_{H+}(x,v)$ exists, then there exists a cone  $C=C(x,v,\delta)$ such that  $\limsup_{y \to x, y \in C} \frac{\|f(y)-f(x)\|}{\|y-x\|} < \infty$. 
        \end{enumerate}
      \end{fact}
   
 We will need the following easy fact (see \cite[Lemma 2.3]{Za2}).
\begin{lemma}\label{hnli}
 Let $X$ be a Banach space, $Y$ a Banach space, $G \subset X$ an open set, $a \in G$, and $f:G \to Y$ a mapping.
  Then the following are equivalent.
  \begin{enumerate}
  \item
    $f'_{H+}(a,0)$ exists,
    \item
    $f'_{H}(a,0)$ exists,
    \item
    $f'_{H}(a,0)=0$,
    \item
    $f$ is Lipschitz at $a$.
    \end{enumerate}
\end{lemma}

 The usual modern definition of the Hadamard derivative is the following:
 
 A continuous linear operator $L: X \to Y$ is said to be the {\it Hadamard derivative} of $f$ at a point $x \in X$ if
  $$ \lim_{t \to 0} \frac{f(x+tv)-f(x)}{t} = L(v)\ \ \  \text{for each}\ \ \ v \in X$$
   and the limit is uniform with respect to $v \in C$, whenever $C\subset X$ is a compact set.
    In this case we set $f'_H(x) := L$. 
   
   The following fact is well-known (see \cite{Sha}):
   
   \begin{lemma}\label{eh}
   Let $X$, $Y$ be Banach spaces, $\emptyset \neq G \subset X$  an open set, $x \in G$, $f: G \to Y$ a mapping and
    $L: X \to Y$ a continuous linear operator. Then the following conditions are equivalent:
    \begin{enumerate}
    \item
    $f'_H(x) = L$,
    \item
    $f'_H(x,v)= L(v)$ for each $v \in X$,
    \item
    if $\vf: [0,1] \to X$ is such that $\vf(0)=x$ and $\vf_+'(0)$ exists, then $(f \circ \vf)_+'(0)= L(\vf_+'(0))$.
    \end{enumerate}
    \end{lemma}

    \begin{definition}
    Let $X$ be a Banach space. We say that $A \subset X$ is {\it directionally porous at a point $x \in X$}, if there exist
     $0 \neq v \in X$, $p>0$  and a sequence  $t_n \to 0$ of positive real numbers such that $B(x+  t_n v, p t_n) \cap A = \emptyset$. (In this case we say that $A$ is  {\it porous at $x$ in  the direction $v$}.)
      
      We say that $A \subset X$ is {\it directionally porous} if $A$ is directionally porous at each point $x \in A$. 
      
      We say that $A \subset X$ is {\it $\sigma$-directionally porous} if it is a countable union of directionally porous
       sets.
       \end{definition}

       We will need the obvious fact that $A$ is not porous at $x$ in the direction $v\neq 0$ if and only if 
       \begin{equation}\label{nepo}
       \text{for each $\omega>0$ there exists $\delta>0$ such that  $B(x+tv, \omega t) \cap A \neq \emptyset$ for
        $0<t<\delta.$}
       \end{equation}
       
       It is easy to see that if, for some $v \in S_X$, $\delta>0$, $r>0$,
       \begin{equation}\label{uhpo}
       \text{$C(x,v,\delta)  \cap B(x,r) \cap A = \emptyset$, then $A$ is porous at $x$ in direction $v$.}
       \end{equation}

Moreover, we will need  the following results from \cite{Za2}.

       \begin{proposition}\label{smerdh}\ $($\cite[Proposition 3.2]{Za2}$)$\ 
 Let $X$ be a separable Banach space, $Y$ a Banach space, $G \subset X$ an open set, and $f: G \to Y$ a mapping. Let $M$ be the set of all
  $x \in G$ at which $f$ is Lipschitz  and there exists $v\in X$ such that     $f'_+(x,v)$ exists but $f'_{H+}(x,v)$
   does not exist. Then $M$ is $\sigma$-directionally porous.
   \end{proposition}

   \begin{proposition}\label{cll}\ (\cite[Proposition 3.1]{Za2})\ 
    Let $X$ be a separable Banach space, $Y$ a Banach space, $G \subset X$ an open set, and $f:G \to Y$ a mapping. Let
     $A$ be the set of all points $x \in G$ for which there exists a cone $C= C(x,v,\delta)$ such that
     $\limsup_{y \to x, y \in C} \frac{\|f(y)-f(x)\|}{\|y-x\|} < \infty$ and $f$ is not Lipschitz at $x$.
      Then $A$  is a  $\sigma$-directionally porous  set.
       \end{proposition}

   \section{Main results}   
    
   \begin{lemma}\label{rozd}
    Let $X$ be a separable Banach space, $Y$ a Banach space, $G \subset X$ an open set, and $f: G \to Y$ a mapping.
     Let $A$ be the set of all $x \in G$ at which $f$ is Lipschitz, and for which there exist $v, w \in X$ such that
      $f'_+(x,v)$ and $f'_+(x,w)$ exist but either $f'_+(x, v-w)$ does not exist or $f_+'(x, v-w) \neq f'_+(x,v)- f'_+(x,w)$. Then $A$ is a $\sigma$-directionally porous set.
    \end{lemma}
 \begin{proof}
 Let $M$ be the $\sigma$-directionally porous set from Proposition \ref{smerdh}. It is sufficient to prove that
  $A^*:= A \setminus M$ is $\sigma$-directionally porous.  To each $x \in A$, choose a corresponding pair
   $v=v_x$, $w=w_x$.  Let $\{d_j: j \in \N\}$ be a dense subset of $X$.
  
  Now consider an arbitrary $x \in A^*$. We can choose $n_x \in \N$ such that
  \begin{equation}\label{enix}
  \|f(y)-f(x)\| \leq n_x \, \|y-x\|\ \ \text{whenever}\ \ \|y-x\| \leq \frac{1}{n_x}.
  \end{equation}
  Further, we can choose $p_x \in \N$ and a sequence  $1> t_i^x \searrow 0$ such that, for each $i \in \N$,
  \begin{equation}\label{poti}
  \left \| \frac{f(x+t_i^x(v_x-w_x)) - f(x)}{t_i^x} - (f'_+(x,v_x) - f'_+(x, w_x))\right\| > \frac{6}{p_x}.
  \end{equation}
  Since $x \notin M$, we have $f'_+(x,w_x) = f'_{H+}(x,w_x)$ and so we can choose $k_x \in \N$ such that
  \begin{equation}\label{powx}
  \left \| \frac{f(x+tz) - f(x)}{t} - f'_+(x, w_x)\right\|< \frac{1}{p_x}\ \ \text{whenever}\ \ \|z-w_x\|< \frac{1}{k_x}\ \ \text{and}\ \ 0< t < \frac{1}{k_x},
  \end{equation}
 and
  \begin{equation}\label{povx}
  \left \| \frac{f(x+tv_x) - f(x)}{t} - f'_+(x, v_x)\right\|< \frac{1}{p_x}\ \ \text{whenever}\ \  0< t < \frac{1}{k_x}.
  \end{equation}
 Finally choose $j_x \in \N$ such that $\|w_x - d_{j_x}\| <  (4 k_x)^{-1}$.
 
 Now, for natural numbers $n, p, k, j$ denote by $A^*_{n,p,k,j}$ the set of all $x \in A^*$ for which
  $n_x=n$, $p_x=p$, $k_x=k$, $j_x=j$.  It is clearly sufficient to prove that, for any fixed quadruple $n,p,k,j$, the set  $S:= A^*_{n,p,k,j}$ is directionally porous.
  
  So fix an arbitrary $x \in S$. 
   Denote $t_i := t_i^x$ and choose $\eta$  such that
   \begin{equation}\label{eta}
   0 < \eta < \min( (pn)^{-1}, (2k)^{-1}).
   \end{equation}
   
   We will prove that $S$ is porous at $x$ in the direction $v_x-w_x$. 
  
  To show this, it is sufficient to prove that there exists $i_0 \in \N$ such that
  \begin{equation}\label{prpr}
  S \cap B(x+t_i(v_x-w_x),\eta t_i) = \emptyset\ \ \text{whenever }\ \ i \geq i_0.
  \end{equation}
  
  To this end, consider $i \in \N$ and $y_i \in S \cap B(x+t_i(v_x-w_x),\eta t_i)$.
  
  First observe that $\|(x + t_i(v_x-w_x))-y_i\| < \eta t_i < \eta < (pn)^{-1}$ and so \eqref{enix} (with $x:=y_i$
   and $y:= x + t_i(v_x-w_x)$) implies
   \begin{equation}\label{yibl}
   \|f(x + t_i(v_x-w_x))-f(y_i)\| \leq n \eta t_i < \frac{t_i}{p}.
     \end{equation}

  Now we will show (and this is the main trick of the proof) that, if $i$ is sufficiently large,
   then the difference of $f'_+(x, w_x)$ and $f'_+(y_i, w_{y_i})$ is ``small'' (see \eqref{auto}). To this end, set $t^*:= (2k)^{-1}$
    and $\xi: = x+ t^* w_x$.  By \eqref{powx} we immediately obtain 
    \begin{equation}\label{xxi}
    \left \| \frac{f(\xi) - f(x)}{t^*} - f'_+(x, w_x)\right\|< \frac{1}{p}.
    \end{equation}
    Further set $z^* := (t^*)^{-1} (\xi-y_i)$. Then
    $$ \|z^*- w_{y_i}\| \leq \|z^*-w_x\| + \|w_x- d_j\|+ \|w_{y_i}-d_j\| \leq \|z^*-w_x\| + \frac{1}{2k}$$
 and
 $$ \|z^*-w_x\| = \|(t^*)^{-1} (\xi-y_i) - (t^*)^{-1}(\xi -x)\|= \|(t^*)^{-1} (x-y_i)\|\leq 2k t_i (\|v_x - w_x\| + \eta).$$
 So there exists $i_1 \in \N$ such that $i \geq i_1$ implies $ \|z^*- w_{y_i}\| < 1/k$, and consequently also
  (by \eqref{powx} with $x: =y_i \in S$)
  \begin{equation}\label{yixi}
   \left \| \frac{f(\xi) - f(y_i)}{t^*} - f'_+(y_i, w_{y_i})\right\| =  \left \| \frac{f(y_i +t^*z^*) - f(y_i)}{t^*} -      f'_+(y_i, w_{y_i}) \right\|< \frac{1}{p}.
   \end{equation}
   Further there exists $i_1 <i_2 \in \N$ such that $i \geq i_2$ implies 
   $$\|x-y_i\|\leq t_i (\|v_x - w_x\| + \eta) <  (2knp)^{-1} < n^{-1}$$
    and consequently also (by \eqref{enix})
   $$ \left\| \frac{f(\xi) - f(x)}{t^*} -  \frac{f(\xi) - f(y_i)}{t^*}\right\| = 2k \|f(y_i) - f(x)\| \leq 2kn \|y_i-x\|
    < \frac{1}{p}.$$
    Using this together with \eqref{xxi} and \eqref{yixi}, we obtain that $i \geq i_2$ implies
    \begin{equation}\label{auto}
     \|f'_+(x, w_x) - f'_+(y_i, w_{y_i})\| < \frac{3}{p}.
    \end{equation}

   Denote  $a_i:= x + t_i v_x$. There exists $i_0>i_2$ such that $i \geq i_0$ implies $t_i < 1/k$, and consequently
    also  (by \eqref{povx}) 
    \begin{equation}\label{xzi}
    \left\| \frac{f(a_i)-f(x)}{t_i}- f'_+(x,v_x)\right\| < \frac{1}{p}.
       \end{equation} 
    Set  $z_i:= (t_i)^{-1} ( a_i -y_i)= (t_i)^{-1} (x+ t_i v_x - y_i)$.  Then
    $$ \|z_i- w_{y_i}\| \leq \|z_i-w_x\| + \|w_x- d_j\|+ \|w_{y_i}-d_j\| \leq \|z_i-w_x\| + \frac{1}{2k}$$
    and
    $$ \|z_i-w_{x}\| = \left\|\frac{x+t_i v_x -y_i}{t_i} -w_x\right\| =
     \left\|\frac{x+t_i(v_x-w_x) -y_i)}{t_i}\right\| < \eta < \frac{1}{2k}. $$ 
   So $ \|z_i- w_{y_i}\| < 1/k$ and $i \geq i_0$ implies $t_i < 1/k$, and consequently (by \eqref{powx} with $x: =y_i)$)
   \begin{equation}\label{ziyi}
   \left \| \frac{f(a_i) - f(y_i)}{t_i} - f'_+(y_i, w_{y_i})\right\| =  \left \| \frac{f(y_i +t_iz_i) - f(y_i)}{t_i} -      f'_+(y_i, w_{y_i}) \right\|< \frac{1}{p}.
   \end{equation}
   Thus, if $i \geq i_0$, then \eqref{yibl}, \eqref{auto}, \eqref{xzi} and \eqref{ziyi} imply
    \begin{multline*}
     \left \| \frac{f(x + t_i(v_x-w_x)) - f(x)}{t_i} - (f'_+(x,v_x) - f'_+(x, w_x))\right\|\\
     \leq
  \left \| \frac{f(y_i) - f(x)}{t_i} - (f'_+(x,v_x) - f'_+(x, w_x))\right\| + \frac{1}{p} \\
  \leq
  \left \| \frac{f(y_i) - f(x)}{t_i} - (f'_+(x,v_x) - f'_+(y_i, w_{y_i}))\right\| + \frac{4}{p} \\ =
  \left \| \left(\frac{f(a_i) - f(x)}{t_i} - f'_+(x,v_x)\right)+ \left(f'_+(y_i, w_{y_i}) - \frac{f(a_i) - f(y_i)}{t_i}\right)  \right\| + \frac{4}{p} < \frac{6}{p},
  \end{multline*}
 which contradicts \eqref{poti}. So we have proved \eqref{prpr}. 
 \end{proof}

   \begin{lemma}\label{uzav}
    Let $X$, $Y$ be Banach spaces, $G \subset X$ an open set,  $f: G \to Y$ a mapping. For each $x \in G$ denote
     by $U_x$ the set of all $v \in X$ such that $f'_+(x,v)$ exists. Then the set $B$ of all $x \in G$ such that
     $f$ is Lipschitz at $x$ and $U_x$ is not closed   is $\sigma$-directionally porous.
       \end{lemma}
 \begin{proof}   
  For each $k \in \N$, set
   $$B_k:= \{x \in B:\ \|f(y)-f(x)\| < k \|y-x\|\ \ \text{ whenever}\ \ \|y-x\| <1/k\}.$$
   It is clearly sufficient to prove that each $B_k$ is directionally porous. So suppose that $k \in \N$
    and $x \in B_k$ are given. Since  $U_x$ is not closed, it is easy to see (using Fact \ref{f}(iv)) that we can find $v \in S_X \setminus U_x$
    and a sequence $v_n \to v$ with $v_n \in S_X \cap U_x$.
    
    We will show that 
    \begin{equation}\label{vesv}
    \text{$B_k$ is porous at $x$ in the direction $v$.}
    \end{equation}
    Suppose, to the contrary,  that \eqref{vesv} does not hold. We will obtain a contradiction by proving that $v \in U_x$,
     i.e., that  $\lim_{t\to 0+} t^{-1} (f(x+tv)- f(x))$ exists. Since $Y$ is a complete space, it is sufficient
      to prove that for each $\ep>0$ there exists $\delta>0$ such that
      \begin{equation}\label{bc}
      \left\|\frac{f(x+t_1v)-f(x)}{t_1} - \frac{f(x+t_2v)-f(x)}{t_2}\right\| < \ep\ \ \text{whenever}\ \ 0<t_1<t_2<\delta.
      \end{equation}
    So, let $1>\ep>0$ be given. Since \eqref{vesv} does not hold,  by \eqref{nepo} we can find $\delta_1 >0$ such that
     \begin{equation}\label{hust}
     \text{$B(x+tv, \ep t/12k) \cap B_k \neq \emptyset$  whenever  $0< t < \delta_1$.}
     \end{equation}
     Now choose $n \in \N$ such that $\|v - v_n\| <  \ep/12k$.
      Since $v_n \in U_x$, we can find $\delta_2 >0$ such that
     \begin{equation}\label{bc2}
      \left\|\frac{f(x+t_1v_n)-f(x)}{t_1} - \frac{f(x+t_2v_n)-f(x)}{t_2}\right\| <  \frac{\ep}{3}\ \ \text{whenever}\ \ 0<t_1<t_2<\delta_2.
      \end{equation}
      Put $\delta: = \min(\delta_1, \delta_2)$. It is sufficient to prove \eqref{bc}. To this end, let arbitrary numbers
      $0<t_1<t_2<\delta$ be given. By \eqref{bc2}, we have
      \begin{equation}\label{bc3}
       \left\|\frac{f(x+t_1v_n)-f(x)}{t_1} - \frac{f(x+t_2v_n)-f(x)}{t_2}\right\| <  \frac{\ep}{3}.
       \end{equation}
       By \eqref{hust}, we can choose points $y_1 \in B(x+t_1v, \ep t_1/12 k) \cap B_k$ and  $y_2 \in B(x+t_2v, \ep t_2/12 k) \cap B_k$.
        Observe that
      $$\|y_1 - (x+t_1v)\|< t_1 \ep/12k<1/k, \ \ \ 
       \|(x+t_1v)-(x+t_1v_n)\|= t_1\|v-v_n\| <  t_1 \ep/12 k,$$
        and therefore  $\|y_1 - (x+t_1v_n)\|< t_1 \ep/6 k< 1/k$.
      Since $y_1 \in B_k$, we obtain 
      $$\|f(y_1)- f(x+t_1v)\| < k t_1 \ep/12 k = t_1\ep/12, 
      \ \ \ \|f(y_1)- f(x+t_1v_n)\| < k t_1 \ep/6k = t_1\ep/6.$$ Consequently
      \begin{equation}\label{odt1}   
      \|f(x+t_1v)- f(x+t_1v_n)\| < 
        t_1\ep/3.
        \end{equation}
        By the same way we obtain
        \begin{equation}\label{odt2}   
      \|f(x+t_2v)- f(x+t_2v_n)\| <
        t_2\ep/3.
        \end{equation}
     So,  for $i\in \{1,2\}$ we obtain
     \begin{equation}\label{vvn}
      \left\|\frac{f(x+t_iv_n)-f(x)}{t_i} - \frac{f(x+t_iv)-f(x)}{t_i}\right\| <  \frac{\ep}{3}.
      \end{equation}
       Using  \eqref{bc3} and \eqref{vvn} we obtain
       $$\left\|\frac{f(x+t_1v)-f(x)}{t_1} - \frac{f(x+t_2v)-f(x)}{t_2}\right\| < \ep$$
        which completes the proof of \eqref{bc}.
      \end{proof}
      
      \begin{theorem}\label{obyc}
    Let $X$ be a separable Banach space, $Y$ a Banach space, $G \subset X$ an open set, and $f: G \to Y$ a mapping.
     Then there exists a  $\sigma$-directionally porous set $C \subset G$ such that if $x \in G \setminus C$ and
      $f$ is Lipschitz at $x$, then the set  $U_x$ of those directions $u \in X$ in which the one-sided derivative $f'_+(x,u)$ exists is a closed linear subspace of $X$. Moreover, the mapping $u \mapsto f'_+(x,u)$ is linear on $U_x$. 
           \end{theorem}
     \begin{proof}
     Let $A$ be the set from Lemma \ref{rozd} and $B$ be the set from Lemma \ref{uzav}. We will show that it is sufficient to set $C:= A \cup B$. Obviously, $C$ is  $\sigma$-directionally porous. 
     
     Now consider an arbitrary $x \in G \setminus C = G \setminus (A\cup B)$ at which $f$ is Lipschitz. 
     It is obvious that $0 \in U_x$, $f'_+(x,0)=0$ and (see Fact \ref{f}(iv))
     \begin{equation}\label{poho}
     \text{if $u \in U_x$ and $t \geq 0$, then $tu \in U_x$ and $f'_+(x,tu) = t f'_+(x,u)$.}
    \end{equation}
    If $w \in U_x$, set $v:=0$ and obtain (since $x \notin A$)
    \begin{equation}\label{opac}
    f'_+(x,-w) = f'_+(x,0) - f'_+(x,w) = - f'_+(x,w).
    \end{equation}
    Obviously, \eqref{poho} and \eqref{opac} imply that
    \begin{equation}\label{homo}
     \text{if $u \in U_x$ and $t \in \R$, then $tu \in U_x$ and $f'_+(x,tu) = t f'_+(x,u)$.}
    \end{equation}
    Further consider arbitrary vectors $v, u \in U_x$. Setting $w:= -u$ and using $x \notin A$ and \eqref{opac},
     we obtain
     \begin{equation}\label{adit}
     f'_+(x, v+u) = f'_+(x,v-w)= f'_+(x,v) - f'_+(x,w) = f'_+(x,v) + f'_+(x,u).
     \end{equation}
     Thus \eqref{homo} and \eqref{adit} imply that $U_x$ is a linear subspace of $X$ and the mapping $u \mapsto f'_+(x,u)$ is linear on $U_x$.  Since $x \notin B$, we have that
      $U_x$ is closed.
     \end{proof}
     
     Theorem \ref{obyc} and Fact \ref{f}(i),(v) immediately imply the following corollaries.
     
     \begin{corollary}\label{liga} 
     Let $X$ be a separable Banach space, $Y$ a Banach space, $G \subset X$ an open set, and $f: G \to Y$ an arbitrary mapping. Then the following implication holds at each point $x \in G$ except a $\sigma$-directionally porous set:
       
 If $f$ is Lipschitz at $x$ and the one-sided directional derivative $f'_{+}(x,u)$ exists in all directions $u$ from a set $S_x \subset X$
  whose linear span is dense in $X$, then $f$ is G\^ ateaux differentiable at $x$.
          \end{corollary}
     \begin{corollary}\label{plux}
     Let $X$ be a separable Banach space, $Y$ a Banach space, $G \subset X$ an open set, and $f: G \to Y$ a pointwise Lipschitz mapping. Then there exists a $\sigma$-directionally porous set $A \subset G$ such that for every $x \in G \setminus A$
 the set $U_x$ of those directions $u \in X$ in which the one-sided derivative $f'_+(x,u)$ exists is a closed linear subspace of $X$. Moreover, the mapping $u \mapsto f'_+(x,u)$ is linear on $U_x$. 
      \end{corollary}
       \begin{corollary}\label{plga}
        Let $X$ be a separable Banach space, $Y$ a Banach space, $G \subset X$ an open set, and $f: G \to Y$ a pointwise Lipschitz mapping. Then the following implication holds at each point $x \in G$ except a $\sigma$-directionally porous set:
          
 If  the one-sided directional derivative $f'_{+}(x,u)$ exists in all directions $u$ from a set $S_x \subset X$
  whose linear span is dense in $X$, then $f$ is G\^ ateaux differentiable at $x$.
       \end{corollary}
      
   Our main result on Hadamard derivatives is the following.
   \begin{theorem}\label{had}
     Let $X$ be a separable Banach space, $Y$ a Banach space, $G \subset X$ an open set, and $f: G \to Y$ a mapping.
     Then there exists a  $\sigma$-directionally porous set $D \subset G$ such that if $x \in G \setminus D$,  then 
     either the set  $V_x$ of those directions $u \in X$ in which the one-sided Hadamard derivative $f'_{H+}(x,u)$ exists is  an empty set or $V_x$ is a closed linear subspace of $X$ and the mapping $u \mapsto f'_{H+}(x,u)$ is linear on $V_x$. 
      \end{theorem}
      \begin{proof}
      Let $C$ be the set from Theorem \ref{obyc}, $M$ the set from Proposition \ref{smerdh} and $A$ the set from Proposition
       \ref{cll}. Set $D:= C \cup M \cup A$. Then $D \subset G$ is a $\sigma$-directionally porous set. Now suppose 
        that a point $x \in G \setminus D$ is given and $V_x \neq \emptyset$. 
        
      Choose $w \in V_x$.  If $w=0$, then we obtain that $f$ is Lipschitz at $x$ by Lemma \ref{hnli}. If $w \neq 0$, then we set $v: = w/\|w\|$ and apply Fact \ref{f}(iv),(vii). We obtain that there exists a cone $C:= C(x,v, \delta)$
       such that $\limsup_{y \to x, y \in C} \frac{\|f(y)-f(x)\|}{\|y-x\|} < \infty$.
       Since $x \notin A$, we obtain
         that  $f$ is Lipschitz at $x$ also in this case.
         
         Since $x \notin M$, we obtain that $f'_+(x,u)=f'_{H+}(x,u)$ whenever one of these two derivatives exists.
          So, since $f$ is Lipschitz at $x$ and  $x \notin C$, we obtain our assertion.
      \end{proof}
      
      \begin{remark}\label{equi}
      If we have in our disposal Proposition \ref{smerdh}, Proposition \ref{cll} and Lemma \ref{hnli}, then we can consider our two main theorems as ``equivalent''. Indeed, we have already shown how Theorem \ref{obyc} (and the above mentioned propositions) easily implies Theorem \ref{had}. Further,  Theorem \ref{had} and  Proposition \ref{cll} almost  immediately
       imply Theorem \ref{obyc}.
           \end{remark}

      Using  Theorem \ref{had}, Fact \ref{f}(i),(vi) and Lemma \ref{eh}, we immediately obtain the following result, which improves and generalizes
      \cite[Theorem 3.7(b)]{Io} (see Remark \ref{dom} below).
      \begin{corollary}\label{hadcom}
       Let $X$ be a separable Banach space, $Y$ a Banach space, $G \subset X$ an open set, and $f: G \to Y$ an arbitrary mapping. Then the following implication holds at each point $x \in G$ except a $\sigma$-directionally porous set:
         
 If the one-sided Hadamard derivative $f'_{H+}(x,u)$ exists in all directions $u$ from a set $S_x \subset X$
  whose linear span is dense in $X$, then $f$ is Hadamard differentiable at $x$.
            \end{corollary}
      \begin{remark}\label{dom}
      Ioffe's \cite[Theorem 3.7(b)]{Io} works with $f: X \to [-\infty,\infty]$ and  $\dom (f):= \{ x \in X:\ 
       |f(x)| < \infty\}$. Denote by $B$ the set of all $x \in \dom (f) \setminus  \interior(\dom(f))$ for which there exists $0 \neq v \in X$ such that  $f'_{H+}(x,v) := \lim_{z \to v, t \to 0+} \frac{f(x+tz)-f(x)}{t}$ exists and is finite. If $x \in B$, then (cf. Fact \ref{f}(vii))  there exists a cone  $C=C(x,v,\delta)$ such that  $\limsup_{y \to x, y \in C} \frac{|f(y)-f(x)|}{\|y-x\|} < \infty$, which implies that $B \subset X \setminus \interior(\dom(f))$ is directionally porous at $x$ (see \eqref{uhpo}). So $B$ is a directionally porous set. 
       
       This simple observation shows that Corollary \ref{hadcom} remains true, if we work with  $f: X \to [-\infty,\infty]$ ( and $f'_{H+}(x,u)$ is, by definition, finite). In other words, the assumption of \cite[Theorem 3.7(b)]{Io} that the  linear span of $S_x$ equals to $X$ can be relaxed to the assumption that the
   linear span of $S_x$ is dense in $X$.
      \end{remark}

 \end{document}